\documentclass[11pt, leqno]{article}

\usepackage{amsmath}
\usepackage{amsthm}
\usepackage{amssymb}

\theoremstyle{plain}
\newtheorem{theorem}{Theorem}[section]
\newtheorem{lemma}[theorem]{Lemma}

\theoremstyle{definition}

\theoremstyle{remark}
\newtheorem{remark}[theorem]{Remark}

\numberwithin{equation}{section}

\begin{document}

\title{\textbf{An operator-theoretical treatment \\of the Maskawa-Nakajima equation \\in the massless abelian gluon model}}

\author{Shuji Watanabe\\
Division of Mathematical Sciences\\
Graduate School of Engineering, Gunma University\\
4-2 Aramaki-machi, Maebashi 371-8510, Japan\\
Email: shuwatanabe@gunma-u.ac.jp}

\date{}

\maketitle

\begin{abstract}
The Maskawa-Nakajima equation has attracted considerable interest in elementary particle physics. From the viewpoint of operator theory, we study the Maskawa-Nakajima equation in the massless abelian gluon model. We first show that there is a nonzero solution to the Maskawa-Nakajima equation when the parameter $\lambda$ satisfies $\lambda>2$. Moreover, we show that the solution is infinitely differentiable and strictly decreasing. We thus conclude that the massless abelian gluon model generates the nonzero quark mass spontaneously and exhibits the spontaneous chiral symmetry breaking when $\lambda>2$. We next show that there is a unique solution $0$ to the Maskawa-Nakajima equation when $0<\lambda<1$, from which we conclude that each quark remains massless and that the model realizes the chiral symmetry when $0<\lambda<1$.

\medskip


\medskip

\end{abstract}


\section{Introduction and preliminaries}

The Maskawa-Nakajima equation has attracted considerable interest in elementary particle physics. The equation is applied to many models such as a massive abelian gluon model \cite{maskawa-nakajima-one, maskawa-nakajima-two, kondo, kugo-nakajima}, a massless abelian gluon model \cite{fukuda-kugo}, a QCD (quantum chromodynamics)-like model \cite{higashijima, kondo-two}, a technicolor model \cite{yamawaki} and a top quark condensation model \cite{miransky-tanabashi-yamawaki}. See also \cite{jackiwjohnson, cornwallnorton} for a technicolor model. The reason why Professor Maskawa reconsider the spontaneous chiral symmetry breaking in a renormalizable model of strong interaction and so on is mentioned in his Nobel lecture.

The Maskawa-Nakajima equation in the massless abelian gluon model is the following nonlinear integral equation \cite{maskawa-nakajima-one, maskawa-nakajima-two, fukuda-kugo, kugo-nakajima} :
\begin{equation}\label{eq:mnequation}
u(x)=\frac{\,\lambda\,}{2}\int_{\varepsilon}^{\Lambda}
\frac{1}{\, y+x+|y-x| \,}\,
\frac{y\, u(y)}{\, y+u(y)^2\,}\, dy,
\qquad \varepsilon \leq x \leq \Lambda,
\end{equation}
where $\varepsilon>0$ is the infrared cutoff, $\Lambda>0$ is the ultraviolet cutoff and $\varepsilon<\Lambda$. Here, $\displaystyle{ \lambda=\frac{3a^2}{\, 4\pi^2\,}>0 }$ is a parameter with $a > 0$ called the gauge coupling constant. When $\lambda>2$, we assume that $\varepsilon$ and $\Lambda$ satisfy \eqref{eq:varepsilonlambda} below.

The adopted gauge group of the massless abelian gluon model is $U(1)$, an abelian group, and the model describes the gauge interactions among quarks and massless gluons. The solution $u$ to the Maskawa-Nakajima equation \eqref{eq:mnequation} is the mass function of the quark, and it is known that the quark mass breaks the symmetry called the chiral symmetry of the model. So, if there is a nonzero solution to the Maskawa-Nakajima equation \eqref{eq:mnequation}, then the nonzero quark mass is generated spontaneously and the massless abelian gluon model is said to exhibit the spontaneous chiral symmetry breaking; if there is a unique solution $0$ to the Maskawa-Nakajima equation \eqref{eq:mnequation}, then each quark remains massless and the massless abelian gluon model is said to realize the chiral symmetry. Maskawa and Nakajima \cite{maskawa-nakajima-one, maskawa-nakajima-two} showed the following: If $\lambda >0$ is large (strong coupling), then the nonzero quark mass is generated spontaneously because of strong coupling, and hence, the chiral symmetry breaks spontaneously. On the other hand, if $\lambda >0$ is small (weak coupling), then the quark mass is identically zero, and hence, the chiral symmetry is preserved.

The present author thinks that the Maskawa-Nakajima equation plays a role similar to that of the BCS gap equation in the BCS model \cite{bcs, bogoliubov} for superconductivity. If there is a nonzero solution to the BCS gap equation, then the BCS model exhibits the spontaneous breaking of the $U(1)$ symmetry; if there is a unique solution $0$ to the BCS gap equation, then the BCS model realizes the $U(1)$ symmetry. The existence and uniqueness of the solution to the gap equation as well as its properties are studied in \cite{bls, odeh, billardfano, vansevesant, fhns, hhss, haizlseiringer, watanabe, watanabe-two}.

Let $\displaystyle{ w \in C[\varepsilon,\, \Lambda] }$ be
\begin{equation}\label{eq:functionw}
w(x)=\frac{\, 4\varepsilon\,}{\lambda}\sqrt\frac{\varepsilon}{\,\Lambda x\,}\, ,\qquad \varepsilon \leq x \leq \Lambda.
\end{equation}

We consider the following subset of the Banach space $C[\varepsilon,\, \Lambda]$:
\begin{equation}\label{eq:setv}
V=\left\{ u \in C[\varepsilon,\, \Lambda]: \; w(x) \leq u(x) \leq \frac{\,\lambda\,}{4}\sqrt{\Lambda} \;\; \mbox{at} \;\; \mbox{all} \;\; x \in [\varepsilon,\, \Lambda] \right\}.
\end{equation}

For $u \in V$, we define a nonlinear integral operator $A$ by
\begin{equation}\label{eq:operatora}
Au(x)=\frac{\,\lambda\,}{2}\int_{\varepsilon}^{\Lambda}
\frac{1}{\, y+x+|y-x| \,}\,
\frac{y\, u(y)}{\, y+u(y)^2\,}\, dy, \qquad \varepsilon \leq x \leq \Lambda.
\end{equation}
Then $Au(x)$ is well-defined at every $x \in [\varepsilon,\, \Lambda]$. Note that $Au(x)$ coincides with the right side of \eqref{eq:mnequation}. So we look for a fixed point of the nonlinear integral operator $A$. Note also that $Au(x)$ is rewritten as
\begin{equation}\label{eq:operatoraa}
Au(x)=\frac{\,\lambda\,}{4} \left\{
\frac{1}{\, x\,} \int_{\varepsilon}^x \frac{y\, u(y)}{\, y+u(y)^2\,}\, dy
+\int_x^{\Lambda} \frac{u(y)}{\, y+u(y)^2\,}\, dy \right\},
\qquad \varepsilon \leq x \leq \Lambda.
\end{equation}

Assume that the function $u$ is a fixed point of the operator $A$. Since $u \in V$ is continuous on $[\varepsilon,\, \Lambda]$, it follows that $u$ is smooth on $[\varepsilon,\, \Lambda]$. So, by \eqref{eq:operatoraa},
\begin{equation}\label{eq:de-one}
x^2 u''(x)+2x u'(x)+\frac{\,\lambda\,}{4}
\frac{xu(x)}{\, x+u(x)^2\,}=0, \qquad \varepsilon \leq x \leq \Lambda.
\end{equation}
Assume also that $u(x)>0$ is small enough at $x$ large enough. It then follows from \eqref{eq:de-one} that at $x$ large enough,
\begin{equation}\label{eq:de-two}
x^2 u''(x)+2x u'(x)+\frac{\,\lambda\,}{4}u(x)=0.
\end{equation}
See Kugo and Nakajima \cite{kugo-nakajima}. Studying the solution to the ordinary differential equation \eqref{eq:de-two} leads to the conclusions that the massless abelian gluon model exhibits the spontaneous chiral symmetry breaking for $\lambda>1$ and that the model realizes the chiral symmetry for $0<\lambda<1$ (see \cite{fukuda-kugo, kugo-nakajima}). Note that \eqref{eq:de-two} is obtained under the assumption that $u(x)>0$ is small enough at $x$ large enough. But one does not know whether or not $u(x)>0$ is small enough at $x$ large enough. It may happen that $u(x)$ is very large at $x$ large enough. So we address this problem from the viewpoint of operator theory without the assumption that $u(x)>0$ is small enough at $x$ large enough.

On the other hand, replacing $u(x)$ of \eqref{eq:mnequation} by $\displaystyle{ u(x)=\frac{\psi(x)}{\, \sqrt{x+\varepsilon} \,} }$ and letting $\Lambda$ tend to infinity change \eqref{eq:mnequation} into
\begin{equation}\label{eq:mnequationprime}
\psi(x)=\frac{\,\lambda\,}{2} \sqrt{x+\varepsilon}
\int_{\varepsilon}^{\infty} \frac{1}{\, y+x+|y-x| \,}\,
\frac{y}{\, \sqrt{y+\varepsilon}\,} \,
\frac{\psi(y)}{\, y+
\displaystyle{ \frac{\psi(y)^2}{\, y+\varepsilon \,} }\,}\, dy,
\qquad \varepsilon \leq x < \infty.
\end{equation}

\begin{remark}\label{rmk:interest}
Because of our mathematical interest, we let $\Lambda \to \infty$ in \eqref{eq:mnequationprime}, and so we consider the Banach space $B^0[\varepsilon,\, \infty)$ in \eqref{eq:setw} below. We can similarly deal with the case where $\Lambda$ remains finite. See Remark \ref{rmk:w1} below.
\end{remark}

We consider the Banach space $B^0[\varepsilon,\, \infty)$ consisting of all the bounded and continuous functions on $[\varepsilon,\, \infty)$, and deal with the following subset:
\begin{equation}\label{eq:setw}
W=\left\{ \psi\in B^0[\varepsilon,\, \infty): \; \psi(x) \geq 0 \;\; \mbox{at} \;\; \mbox{all} \;\; x \in [\varepsilon,\,\infty) \right\}.
\end{equation}

For $\psi \in W$, we define another nonlinear integral operator $B$ by
\begin{equation}\label{eq:operatorb}
B\psi(x)=\frac{\,\lambda\,}{2} \sqrt{x+\varepsilon}
\int_{\varepsilon}^{\infty} \frac{1}{\, y+x+|y-x| \,}\,
\frac{y}{\, \sqrt{y+\varepsilon}\,} \,
\frac{\psi(y)}{\, y+
\displaystyle{ \frac{\psi(y)^2}{\, y+\varepsilon \,} }\,}\, dy,
\qquad \varepsilon \leq x < \infty.
\end{equation}
Then $B\psi(x)$ is well-defined at every $x \in [\varepsilon,\,\infty)$.  Note that $B\psi(x)$ coincides with the right side of \eqref{eq:mnequationprime}. So we again look for a fixed point of the nonlinear integral operator $B$.

The paper proceeds as follows. In section 2 we state our main results without proof. In sections 3 and 4 we prove our main results.

\section{Main results}

We first deal with the case where $\lambda>2$. For $\lambda>2$, let $\varepsilon$ and $\Lambda$ satisfy
\begin{equation}\label{eq:varepsilonlambda}
\frac{\varepsilon}{\,\Lambda\,} \leq \min\left( \frac{1}{\, 16\,},\,\left(
\frac{\, \sqrt{\lambda^2+128(\lambda-2)}-\lambda\,}{64} \right)^2 \right).
\end{equation}

\begin{theorem}\label{thm:fxd-pt-A}
Assume $\lambda>2$. Let $\varepsilon$ and $\Lambda$ satisfy \eqref{eq:varepsilonlambda}. Let $A$ be as in \eqref{eq:operatora} and $V$ as in \eqref{eq:setv}.

\noindent {\rm (a)} \   The nonlinear integral operator $A: V \to V$ has at least one fixed point $u_0 \in V$. Consequently, $u_0$ is continuous on $[\varepsilon,\, \Lambda]$ and $(0<)\, w(x) \leq u_0(x) \leq \lambda\sqrt{\Lambda}/4$ at all $x \in [\varepsilon,\,\Lambda]$. Hence the massless abelian gluon model 
generates the nonzero quark mass spontaneously, and so the model exhibits the spontaneous chiral symmetry breaking. Moreover, each fixed point $u_0 \in V$ is strictly decreasing on $[\varepsilon,\, \Lambda]$, and satisfies
\[
u_0 \in C^{\infty}[\varepsilon,\, \Lambda], \qquad u_0'(\varepsilon)=0.
\]

\noindent {\rm (b)} \   Let $u_0  \in V$ be a fixed point of the nonlinear integral operator $A: V \to V$ above. If $u_0(x) \leq \sqrt{x}$ at all $x \in [\varepsilon,\, \Lambda]$, then $A: V \to V$ has a unique fixed point $u_0 \in V$.
\end{theorem}

We then deal with the case where $0<\lambda<1$.

\begin{theorem}\label{thm:fxd-pt-B}
Assume $0<\lambda<1$. Let $B$ be as in \eqref{eq:operatorb} and $W$ as in \eqref{eq:setw}. Then the nonlinear integral operator $B: W \to W$ has a unique fixed point $0 \in W$. Consequently, each quark remains massless and the massless abelian gluon model realizes the chiral symmetry.
\end{theorem}

\begin{remark}\label{rmk:w1}
Because of our mathematical interest, we let $\Lambda \to \infty$ in \eqref{eq:mnequationprime} and we consider the Banach space $B^0[\varepsilon,\, \infty)$ in \eqref{eq:setw}, as mentioned in Remark \ref{rmk:interest}. Indeed, when $0<\lambda<1$, we can similarly deal with the case where $\Lambda$ remains finite. In this case, we define the operator $C$ on the set
\[
W_1=\left\{ \psi\in C[\varepsilon,\, \Lambda]: \; \psi(x) \geq 0 \;\; \mbox{at} \;\; \mbox{all} \;\; x \in [\varepsilon,\,\Lambda] \right\}
\]
instead of the set $W$ (see \eqref{eq:setw}):
\[
C\psi(x)=\frac{\,\lambda\,}{2} \sqrt{x+\varepsilon}
\int_{\varepsilon}^{\Lambda} \frac{1}{\, y+x+|y-x| \,}\,
\frac{y}{\, \sqrt{y+\varepsilon}\,} \,
\frac{\psi(y)}{\, y+
\displaystyle{ \frac{\psi(y)^2}{\, y+\varepsilon \,} }\,}\, dy,
\quad \psi \in W_1\,, \quad \varepsilon \leq x \leq \Lambda.
\]
We can similarly show that the nonlinear integral operator $C: W_1 \to W_1$ has a unique fixed point $0 \in W_1$, and hence we conclude that the massless abelian gluon model realizes the chiral symmetry when $0<\lambda<1$.
\end{remark}

\section{Proof of Theorem \ref{thm:fxd-pt-A}}

In this section we prove Theorem \ref{thm:fxd-pt-A} in a sequence of lemmas. So, throughout this section, we assume $\lambda>2$ and suppose that \eqref{eq:varepsilonlambda} holds.

We study some properties of the operator $A$ given by \eqref{eq:operatora}.

\begin{lemma}
Let $w$ be as in \eqref{eq:functionw}. Then $Aw(x)>w(x)$ at all $x \in [\varepsilon,\, \Lambda]$.
\end{lemma}

\begin{proof}\quad Since $y^2+\frac{16\varepsilon^3}{\, \lambda^2\Lambda\,}<
\left( 1+\frac{\, 16\varepsilon\,}{\Lambda} \right)y^2$ and $\sqrt{\frac{\,\varepsilon\,}{x}}+\sqrt{\frac{x}{\,\Lambda\,}} \leq  1+\sqrt{ \frac{\varepsilon}{\,\Lambda\,} }$, it follows from \eqref{eq:operatoraa} that
\begin{eqnarray*}
Aw(x) &=& w(x)\frac{\,\lambda\,}{4} \left\{ 
\frac{1}{\, \sqrt{x}\,}\int_{\varepsilon}^x
\frac{y^{3/2}}{\, y^2+\frac{16\varepsilon^3}{\, \lambda^2\Lambda\,}\,}\, dy
+\sqrt{x}\int_x^{\Lambda}
\frac{y^{1/2}}{\, y^2+\frac{16\varepsilon^3}{\, \lambda^2\Lambda\,}\,}
\, dy \right\} \\
\, &>& w(x)\frac{\lambda}{\, 1+\frac{\, 16\varepsilon\,}{\Lambda}\,}
\left\{ 1-\frac{1}{\, 2\,}
\left( \sqrt{\frac{\,\varepsilon\,}{x}}+\sqrt{\frac{x}{\,\Lambda\,}}
\right) \right\} \\
\, &\geq& w(x)\frac{ \lambda\left( 1-\sqrt{ \frac{\varepsilon}{\,\Lambda\,} }
\right)} {\, 2\left( 1+\frac{\, 16\varepsilon\,}{\Lambda} \right)\,}.
\end{eqnarray*}
The inequality \   $\displaystyle{
\frac{ \lambda\left( 1-\sqrt{ \frac{\varepsilon}{\,\Lambda\,} } \right)}
{\, 2\left( 1+\frac{\, 16\varepsilon\,}{\Lambda} \right)\,} \geq 1
}$ follows from \eqref{eq:varepsilonlambda}.
\end{proof}

\begin{lemma}\label{lm:important}
Let $w$ be as in \eqref{eq:functionw} and let $u \in V$. Then \   $\displaystyle{ \frac{u(x)}{\, x+u(x)^2\,} \geq \frac{w(x)}{\, x+w(x)^2\,} }$, and hence $Au(x) \geq Aw(x)$ at all $x \in [\varepsilon,\, \Lambda]$.
\end{lemma}

\begin{proof}
Let $u \in V$ and let $y \in [\varepsilon,\, \Lambda]$ be fixed. Since $y \geq \varepsilon$, it follows that
\[
u(y) \leq \frac{\,\lambda\,}{4}\sqrt{\Lambda} \leq \frac{\,\lambda\,}{4}\sqrt{\Lambda}
\left( \frac{y}{\,\varepsilon\,} \right)^{3/2}=\frac{y}{\, w(y)\,}.
\]
Hence $y \geq u(y)w(y)$. If $u(y)>w(y)$, then $y\{ u(y)-w(y)\} \geq u(y)w(y)\{ u(y)-w(y)\}$.
Therefore,
\begin{equation}\label{eq:inequality}
\frac{u(y)}{\, y+u(y)^2\,} \geq \frac{w(y)}{\, y+w(y)^2\,}.
\end{equation}
The inequality \eqref{eq:inequality} holds true even if $u(y)=w(y)$.
Hence \eqref{eq:inequality} holds true for all $u \in V$ and at all $y \in [\varepsilon,\, \Lambda]$. Thus
\[
Au(x) \geq \frac{\,\lambda\,}{2}\int_{\varepsilon}^{\Lambda}
\frac{1}{\, y+x+|y-x| \,}\, \frac{y\, w(y)}{\, y+w(y)^2\,}\, dy=Aw(x).
\]
\end{proof}

The two lemmas just above are our key lemmas in this section.

\begin{lemma}\label{lm:uniformbound}
If $u \in V$, then $\displaystyle{ Au(x) \leq \frac{\,\lambda\,}{4}\sqrt{\Lambda} }$ at all $x \in [\varepsilon,\, \Lambda]$.
\end{lemma}

\begin{proof}
By \eqref{eq:operatoraa}, \quad $\displaystyle{
Au(x) \leq \frac{\,\lambda\,}{8} \left\{
\frac{1}{\, x\,} \int_{\varepsilon}^x \sqrt{y}\, dy
+\int_x^{\Lambda} \frac{1}{\, \sqrt{y} \,}\, dy \right\}<
\frac{\,\lambda\,}{4}\sqrt{\Lambda} }$.
\end{proof}

\begin{lemma}\label{lm:equicont}
If $u \in V$, then $Au \in C[\varepsilon,\, \Lambda]$.
\end{lemma}

\begin{proof}
Let $\varepsilon \leq x_0 \leq \Lambda$. Then, for $u \in V$,
\[
\left| Au(x)-Au(x_0) \right| \leq \frac{\,\lambda\,}{2}
\int_{\varepsilon}^{\Lambda}
\left| \frac{1}{\, y+x+|y-x| \,}-\frac{1}{\, y+x_0+|y-x_0| \,} \right|
\frac{\,\sqrt{y} \,}{2}\, dy.
\]
The function $\displaystyle{ (x,\, y) \mapsto \frac{1}{\, y+x+|y-x| \,} }$ is uniformly continuous on $[\varepsilon,\, \Lambda]^2$. Hence, for an arbitrary $\varepsilon_1>0$, there is a $\delta>0$ satifying
\[
\left| \frac{1}{\, y+x+|y-x| \,}-\frac{1}{\, y+x_0+|y-x_0| \,} \right| < \varepsilon_1 \, ,\qquad \left| x-x_0 \right| < \delta.
\]
Note that $\delta$ depends neither on $x$, nor on $x_0$, nor on $y$, nor on $u$. Therefore,
\[
\left| Au(x)-Au(x_0) \right| < \frac{\,\lambda\Lambda^{3/2}\,}{6}\varepsilon_1 \,,\qquad \left| x-x_0 \right| < \delta.
\]
\end{proof}

These four lemmas above imply that the set $AV=\left\{ Au: \; u \in V \right\}$ is a subset of $V$.
\begin{lemma}
\quad  $AV \subset V$.
\end{lemma}

\begin{lemma}
The set $AV$ is relatively compact.
\end{lemma}

\begin{proof}
By Lemma \ref{lm:uniformbound}, the set $AV$ is uniformly bounded. As mentioned in the proof of Lemma \ref{lm:equicont}, the $\delta$ does not depend on $u \in V$. Hence the set $AV$ is equicontinuous. The result thus follows from the Ascoli--Arzel$\grave{\mbox{a}}$ theorem.
\end{proof}

\begin{lemma}
The operator $A: V \to V$ is continuous.
\end{lemma}

\begin{proof}
Let $u,\, v \in V$. Then, at every $x \in [\varepsilon,\, \Lambda]$,
\[
\left| Au(x)-Av(x) \right| \leq \frac{\,\lambda\,}{2} \,
\left\| u-v \right\|_{C[\varepsilon,\, \Lambda]}
\int_{\varepsilon}^{\Lambda} \frac{1}{\, y+x+|y-x| \,} \, dy\, .
\]
Here, $\left\| \cdot \right\|_{C[\varepsilon,\, \Lambda]}$ denotes the norm of the Banach space $C[\varepsilon,\, \Lambda]$. Since
\[
\int_{\varepsilon}^{\Lambda} \frac{1}{\, y+x+|y-x| \,} \, dy \leq
\frac{1}{\, 2\,} \left( 1+\ln\frac{\,\Lambda\,}{\varepsilon} \right),
\]
it follows \   $\displaystyle{ 
\left\| Au-Av \right\|_{C[\varepsilon,\, \Lambda]} \leq
\frac{\,\lambda\,}{4} \left( 1+\ln\frac{\,\Lambda\,}{\varepsilon} \right)
\left\| u-v \right\|_{C[\varepsilon,\, \Lambda]} }$.
\end{proof}

The set $V$ is clearly bounded, closed and convex. The Schauder fixed-point theorem (see e.g. Zeidler \cite[p. 61]{zeidler}) thus implies the following.

\begin{lemma}\label{lm:fixedpt}
The operator $A: V \to V$ has at least one fixed point $u_0 \in V$. Consequently, the massless abelian gluon model generates the nonzero quark mass spontaneously, and hence the model exhibits the spontaneous chiral symmetry breaking.
\end{lemma}

\begin{lemma}
Let $u_0 \in V$ be a fixed point of the operator $A: V \to V$ given by Lemma \ref{lm:fixedpt}. Then \  $\displaystyle{ u_0 \in C^{\infty}[\varepsilon,\, \Lambda] }$, \  $\displaystyle{ u_0'(\varepsilon)=0 }$ and $u_0$ is strictly decreasing on $[\varepsilon,\, \Lambda]$.
\end{lemma}

\begin{proof}
Each fixed point $u_0 \in V$ satisfies \eqref{eq:mnequation}, i.e.,
\begin{eqnarray*}
u_0(x)&=&\frac{\,\lambda\,}{2} \int_{\varepsilon}^{\Lambda}
\frac{1}{\, y+x+|y-x| \,}\,\frac{y\, u_0(y)}{\, y+u_0(y)^2\,}\, dy \\
&=& \frac{\,\lambda\,}{4} \left\{
\frac{1}{\, x\,} \int_{\varepsilon}^x \frac{y\, u_0(y)}{\, y+u_0(y)^2\,}\, dy
+\int_x^{\Lambda} \frac{u_0(y)}{\, y+u_0(y)^2\,}\, dy \right\},
\quad \varepsilon \leq x \leq \Lambda.
\end{eqnarray*}
Note that $u_0 \in V$ is continuous and $u_0(x)>0$ at each $x \in [\varepsilon,\,\Lambda]$. We then regard the integrals above as the Riemann integrals. Then $u_0$ is differentiable on $[\varepsilon,\,\Lambda]$:
\[
u_0'(x)=-\frac{\lambda}{\, 4x^2\,} \int_{\varepsilon}^x
\frac{y\, u_0(y)}{\, y+u_0(y)^2\,}\, dy \leq 0 \,, \qquad \varepsilon \leq x \leq \Lambda.
\]
Note that the equality $\displaystyle{ u_0'(x)=0 }$ holds at $x=\varepsilon$ only. We thus see that $u_0'(\varepsilon)=0$, that $u_0$ is strictly decreasing on $[\varepsilon,\, \Lambda]$, and that $u_0$ is infinitely differentiable on $[\varepsilon,\,\Lambda]$.
\end{proof}

\begin{lemma}
Let $u_0  \in V$ be a fixed point of the operator $A: V \to V$ given by Lemma \ref{lm:fixedpt}. If $u_0(x) \leq \sqrt{x}$ at all $x \in [\varepsilon,\, \Lambda]$, then $A: V \to V$ has a unique fixed point $u_0 \in V$.
\end{lemma}

\begin{proof}
Let $v_0 \in V$ be another fixed point of $A: V \to V$. Then there are a number $t$ \   $(0<t<1)$ and a point $x_0 \in [\varepsilon,\,\Lambda]$ such that
\begin{equation}\label{eq:uzerovzero}
u_0(x) \geq t \, v_0(x) \quad (x \in [\varepsilon,\,\Lambda]) \quad \mbox{and} \quad u_0(x_0)=t \, v_0(x_0).
\end{equation}
Note that $\displaystyle{ t \, v_0(x) \leq u_0(x) \leq \sqrt{x} }$ at each fixed $x \in [\varepsilon,\, \Lambda]$. Note also that in case of $u_0(x) \geq v_0(x)$ for all $x \in [\varepsilon,\,\Lambda]$ one can change the roles of $u_0$ and $v_0$.

Since $y \geq u_0(y) \, t \, v_0(y)$ at each fixed $y \in [\varepsilon,\, \Lambda]$, an argument similar to that in the proof of Lemma \ref{lm:important} gives
\[
\frac{u_0(y)}{\, y+u_0(y)^2\,} \geq \frac{t\, v_0(y)}{\, y+t^2\, v_0(y)^2\,}.
\]
Therefore,
\begin{eqnarray*}
u_0(x_0) &=& \frac{\,\lambda\,}{2} \int_{\varepsilon}^{\Lambda}
\frac{1}{\, y+x_0+|y-x_0| \,}\,\frac{y\, u_0(y)}{\, y+u_0(y)^2\,}\, dy \\
&\geq& \frac{\,\lambda\,}{2} \int_{\varepsilon}^{\Lambda}
\frac{1}{\, y+x_0+|y-x_0| \,}\,\frac{y\, t\, v_0(y)}{\, y+t^2\, v_0(y)^2\,}\, dy \, .
\end{eqnarray*}
Hence,
\begin{eqnarray*}
\frac{\,\lambda\,}{2} \int_{\varepsilon}^{\Lambda}
\frac{1}{\, y+x_0+|y-x_0| \,}\,\frac{y\, t\, v_0(y)}{\, y+t^2\, v_0(y)^2\,}\, dy&>& t\,\frac{\,\lambda\,}{2} \int_{\varepsilon}^{\Lambda}
\frac{1}{\, y+x_0+|y-x_0| \,}\,\frac{y\, v_0(y)}{\, y+v_0(y)^2\,}\, dy \\
&=& t \, v_0(x_0),
\end{eqnarray*}
which contradicts $u_0(x_0)=t \, v_0(x_0)$ \   (see \eqref{eq:uzerovzero}).
\end{proof}

Our proof of Theorem \ref{thm:fxd-pt-A} is now complete.

\section{Proof of Theorem \ref{thm:fxd-pt-B}}

In this section we prove Theorem \ref{thm:fxd-pt-B} in a sequence of lemmas. So, throughout this section, we assume $0<\lambda<1$.

\smallskip

An immediate consequence of the definition of the nonlinear integral operator $B$ is the following.

\begin{lemma}\label{lm:four-one}
Let $W$ be as in \eqref{eq:setw}. If $\psi \in W$, then $B\psi(x) \geq 0$ at all $x \in [\varepsilon,\,\infty)$.
\end{lemma}

\begin{lemma}\label{lm:four-two}
If $\psi \in W$, then $B\psi \in C[\varepsilon,\, \infty)$.
\end{lemma}

\begin{proof}
Let $x_0 \in [\varepsilon,\,\infty)$ and let $x \in [\varepsilon,\, 2x_0]$. Then, for $\psi \in W$,
\begin{eqnarray*}
\left| B\psi(x)-B\psi(x_0) \right| &\leq& \frac{\,\lambda\,}{2}
\int_{\varepsilon}^{\infty} \left|
\frac{ \sqrt{x+\varepsilon} }{\, y+x+|y-x| \,}
-\frac{ \sqrt{x_0+\varepsilon} }{\, y+x_0+|y-x_0| \,} \right|
\, \frac{\psi(y)}{\, \sqrt{y+\varepsilon}\,} \, dy \\
&=& I_1+I_2\, ,
\end{eqnarray*}
where
\begin{eqnarray*}
I_1 &=& \frac{\,\lambda\,}{2} \int_{\varepsilon}^R \left|
\frac{ \sqrt{x+\varepsilon} }{\, y+x+|y-x| \,}
-\frac{ \sqrt{x_0+\varepsilon} }{\, y+x_0+|y-x_0| \,} \right|
\, \frac{\psi(y)}{\, \sqrt{y+\varepsilon}\,} \, dy\, , \\
\noalign{\vskip2mm}
I_2 &=& \frac{\,\lambda\,}{2} \int_R^{\infty} \left|
\frac{ \sqrt{x+\varepsilon} }{\, y+x+|y-x| \,}
-\frac{ \sqrt{x_0+\varepsilon} }{\, y+x_0+|y-x_0| \,} \right|
\, \frac{\psi(y)}{\, \sqrt{y+\varepsilon}\,} \, dy\, .
\end{eqnarray*}
Let $\left\| \cdot \right\|_{B^0[\varepsilon,\, \infty)}$ denote the norm of the Banach space $B^0[\varepsilon,\, \infty)$. A straightforward calculation then gives for an arbitrary $\varepsilon_1>0$,
\begin{eqnarray*}
I_2 &\leq& \lambda \, \| \psi \|_{B^0[\varepsilon,\, \infty)} \int_R^{\infty}
\frac{\, \sqrt{2x_0+\varepsilon}\,}{\, y+\varepsilon \,}
\, \frac{1}{\, \sqrt{y+\varepsilon}\,} \, dy \\
&<& 2\lambda \, \| \psi \|_{B^0[\varepsilon,\, \infty)}
\sqrt{ \frac{\, 2x_0+\varepsilon\,}{R} } \\
&<& \frac{\, \varepsilon_1\,}{2}\, .
\end{eqnarray*}
Here, $\displaystyle{ R>\left( \frac{\, 4\lambda \, \| \psi \|_{B^0[\varepsilon,\, \infty)} \sqrt{2x_0+\varepsilon}\,}{\varepsilon_1} \right)^2 }$. We choose an $R$ satifying this inequality and fix it. Let us denote it by $R_0$.

On the other hand, the function $\displaystyle{ (x,\, y) \mapsto \frac{ \sqrt{x+\varepsilon} }{\, y+x+|y-x| \,} }$ is uniformly continuous on $[\varepsilon,\, 2x_0] \times [\varepsilon,\, R_0]$, and hence there is a $\delta$ such that
\[
\left| \frac{ \sqrt{x+\varepsilon} }{\, y+x+|y-x| \,}
-\frac{ \sqrt{x_0+\varepsilon} }{\, y+x_0+|y-x_0| \,} \right| <
\frac{\varepsilon_1}{\, 2\lambda \, \| \psi \|_{B^0[\varepsilon,\, \infty)}\sqrt{R_0}\,}\, , \qquad |x-x_0| < \delta.
\]
Therefore, $I_1<\varepsilon_1/2$. Thus \   $\displaystyle{
\left| B\psi(x)-B\psi(x_0) \right| < \varepsilon_1\, ,
\quad |x-x_0| < \delta }$.
\end{proof}

\begin{lemma}\label{lm:four-three}
Let $\psi \in W$. Then $\left\| B\psi \right\|_{B^0[\varepsilon,\, \infty)} \leq \lambda\,\left\| \psi \right\|_{B^0[\varepsilon,\, \infty)}$. Consequently, $B\psi$ is bounded.
\end{lemma}

\begin{proof}
\[
B\psi(x) \leq \frac{\,\lambda\,}{2}
\left\| \psi \right\|_{B^0[\varepsilon,\, \infty)}
\sqrt{x+\varepsilon} \int_{\varepsilon}^{\infty}
\frac{dy}{\, \left( y+x+|y-x| \right)\sqrt{y+\varepsilon}\,} \, .
\]
Here, a straightforward calculation gives
\begin{eqnarray*}
\int_{\varepsilon}^{\infty} \frac{dy}{\, \left( y+x+|y-x| \right)\sqrt{y+\varepsilon}\,}
&\leq& \int_0^{\infty} \frac{dy}{\, \left( y+x+|y-x| \right)\sqrt{y+\varepsilon}\,} \\
&=& \frac{1}{\,  2x\,} \int_0^x \frac{dy}{\, \sqrt{y+\varepsilon}\,}+
 \frac{1}{\,  2\, } \int_x^{\infty} \frac{dy}{\, y \,\sqrt{y+\varepsilon}\,} \\
&=& \frac{1}{\, \sqrt{x+\varepsilon}+\sqrt{\varepsilon} \,}
+\frac{1}{\, \sqrt{\varepsilon} \,} \ln
\left( \sqrt{ 1+\frac{\, \varepsilon \,}{x}}+
\sqrt{\frac{\, \varepsilon \,}{x}} \right).
\end{eqnarray*}
The result thus follows from Lemma \ref{lm:xepsilon} just below.
\end{proof}

\begin{lemma}\label{lm:xepsilon}
\[
\sqrt{x+\varepsilon} \left\{ 
\frac{1}{\, \sqrt{x+\varepsilon}+\sqrt{\varepsilon} \,}
+\frac{1}{\, \sqrt{\varepsilon} \,} \ln
\left( \sqrt{ 1+\frac{\, \varepsilon \,}{x}}+
\sqrt{\frac{\, \varepsilon \,}{x}} \right) \right\} < 2 \qquad (x \geq \varepsilon).
\]
\end{lemma}

\begin{proof}
Consider the following function $f$ given by 
\[
f(\xi) = \frac{1}{\, \xi+1+\sqrt{\xi+1}\,}+\frac{1}{\,\sqrt{\xi+1}\,}
-\ln \frac{\,\sqrt{\xi+1}+1\,}{\sqrt{\xi}}, \quad \xi \geq 1.
\]
Then $f(1)=1-\ln(\sqrt{2}+1)>0$ and $\displaystyle{ \lim_{\xi \to \infty}f(\xi)=0 }$.
A straightforward calculation gives
\[
f'(\xi)=\frac{1-\sqrt{\xi+1}(\xi-1)}{\,\xi(\xi+1)^{3/2}(\sqrt{\xi+1}+1)^2},
\]
and hence $f(\xi)>0$ at all $\xi \geq 1$. Then
\[
\sqrt{\xi+1}\ln \frac{\,\sqrt{\xi+1}+1\,}{\sqrt{\xi}}<\frac{1}{\, \sqrt{\xi+1}+1\,}+1.
\]
Thus
\[
\frac{\sqrt{\xi+1}}{\, \sqrt{\xi+1}+1\,}+\sqrt{\xi+1}\ln \frac{\,\sqrt{\xi+1}+1\,}{\sqrt{\xi}}<2.
\]
Setting $\xi=x/\varepsilon$ proves the lemma.
\end{proof}

Lemmas \ref{lm:four-one}, \ref{lm:four-two} and \ref{lm:four-three} imply the following.

\begin{lemma}
If $\psi \in W$, then $B\psi \in W$.
\end{lemma}

We now show that the nonlinear integral operator $B: W \to W$ is contractive. The following lemma is our key lemma in this section.

\begin{lemma}
Let $\psi,\, \varphi \in W$. Then
\[
\left\| B\psi-B\varphi \right\|_{B^0[\varepsilon,\, \infty)} \leq
\lambda \left\| \psi-\varphi \right\|_{B^0[\varepsilon,\, \infty)}\, .
\]
Consequently, $B: W \to W$ is contractive.
\end{lemma}

\begin{proof}
Let $X,\, Y \geq 0$. A straightforward calculation gives
\begin{eqnarray*}
\left| \frac{y\, X}{\, y+X^2\,}-\frac{y\, Y}{\, y+Y^2\,} \right| &=&
 y \left| X-Y \right| \frac{ \, |y-XY|\,}{\, (\, y+X^2\, )(\, y+Y^2\, )\,} \\
&\leq& y \left| X-Y \right| \frac{ y+2XY }{\, y^2+y(\, X^2+Y^2\, )\,} \\
&\leq& \left| X-Y \right|.
\end{eqnarray*}
Replacing $X$ by $\psi(y)/\sqrt{y+\varepsilon}$ and $Y$ by $\varphi(y)/\sqrt{y+\varepsilon}$ yields
\begin{eqnarray*}
\left| B\psi(x)-B\varphi(x) \right| &\leq& \frac{\,\lambda\,}{2}
 \sqrt{x+\varepsilon} \int_{\varepsilon}^{\infty}
 \frac{1}{\, y+x+|y-x| \,}\,
 \left| \frac{\psi(y)}{\, \sqrt{y+\varepsilon}\,}
 -\frac{\varphi(y)}{\, \sqrt{y+\varepsilon}\,} \right| \, dy \\
&\leq& \frac{\,\lambda\,}{2} \,
 \left\| \psi-\varphi \right\|_{B^0[\varepsilon,\, \infty)}
 \sqrt{x+\varepsilon} \int_{\varepsilon}^{\infty}
 \frac{dy}{\, \left( y+x+|y-x| \right)\sqrt{y+\varepsilon}\,} \\
&\leq& \lambda \left\| \psi-\varphi \right\|_{B^0[\varepsilon,\, \infty)}
\end{eqnarray*}
by Lemma \ref{lm:xepsilon}, from which the result follows.
\end{proof}

The Banach fixed-point theorem (see e.g. Zeidler \cite[p. 19]{zeidler}) thus implies the following.

\begin{lemma}
The nonlinear integral operator $B: W \to W$ has a unique fixed point $\psi_0 \in W$.
\end{lemma}

Clearly, the element $0 \in W$ is a fixed point of the operator $B: W \to W$ (see \eqref{eq:operatorb}). Moreover, the operator $B: W \to W$ has a unique fixed point $\psi_0 \in W$ by the lemma just above. Therefore, the fixed point $\psi_0 \in W$ is nothing but the element $0 \in W$. We thus have the following.

\begin{lemma}
The nonlinear integral operator $B: W \to W$ has a unique fixed point $0 \in W$. Consequently, each quark remains massless, and hence the massless abelian gluon model realizes the chiral symmetry.
\end{lemma}

Our proof of Theorem \ref{thm:fxd-pt-B} is now complete.

\noindent \textbf{Acknowledgments}

S. Watanabe is supported in part by the JSPS Grant-in-Aid for Scientific Research (C) 24540112. S. Watanabe would like to express his sincere thanks to the reviewer for the valuable comments on the paper.


\end{document}